\documentclass{amsart}
\usepackage{amsmath, amssymb, amsthm}

\usepackage{hyperref}
\usepackage{cleveref}

\newtheorem{dfn}{Definition}[section]
\newtheorem{thm}[dfn]{Theorem}
\newtheorem{lem}[dfn]{Lemma}
\newtheorem{cor}[dfn]{Corollary}
\newtheorem{rem}[dfn]{Remark}
\newtheorem{claim}[dfn]{Claim}

\crefname{dfn}{Definition}{Definitions}
\crefname{thm}{Theorem}{Theorems}
\crefname{lem}{Lemma}{Lemmas}
\crefname{cor}{Corollary}{Corollaries}
\crefname{rem}{Remark}{Remarks}
\crefname{claim}{Claim}{Claims}

\renewcommand{\subset}{\subseteq}
\newcommand{\power}{\wp}
\newcommand{\uphar}{\mathbin{\upharpoonright}}
\DeclareMathOperator{\cf}{cf}
\DeclareMathOperator{\crit}{crit}
\DeclareMathOperator{\ord}{Ord}

\title{Determinacy in the Chang model}
\author{Takehiko Gappo and Grigor Sargsyan}
\date{February 13, 2023}

\begin{document}

\maketitle

\begin{abstract}
Assuming the existence of a certain hod pair with a Woodin cardinal that is a limit of Woodin cardinals, we show that the Chang model satisfies $\mathsf{AD}^+$ in any set generic extensions.
\end{abstract}

\section{Introduction}

The Chang model is the smallest inner model of $\mathsf{ZF}$ including all ordinals and closed under countable sequences.

\begin{dfn}
The Chang model is defined by
\[
\mathsf{CM}=\bigcup_{\alpha\in\ord}L({}^{\omega}\alpha).
\]
We also write $\mathsf{CM}\uphar\alpha=\bigcup_{\beta<\alpha}L({}^{\omega}\beta)$ for any ordinal $\alpha$.
\end{dfn}

By a result of Woodin, assuming the existence of a proper class of Woodin cardinals that are limits of Woodin cardinals, the theory of $\mathsf{CM}$ cannot be changed by forcing.\footnote{See \cite[Chapter 3]{StationaryTower} for the proof of this result.}
However, not much is known about this theory, and the model seems impenetrable. 

Nevertheless, there are two remarkable results on $\mathsf{CM}$.
First, in 90s, Woodin showed that assuming a proper class of Woodin cardinals that are limits of Woodin cardinals, $\mathsf{CM}\models\mathsf{AD}^+$.\footnote{This work will appear as \cite{DetGen}.}
As a corollary of the two aforementioned results of Woodin, we obtain that assuming the existence of a proper class of Woodin cardinals that are limits of Woodin cardinals, the statement $\mathsf{CM}\models\mathsf{AD}^+$ holds in any set generic extensions.

In another remarkable work, in \cite{small_sharp_for_Chang}, Mitchell, building on the forcing technology developed by Gitik, found a surprisingly small $\mathbb{R}$-mouse which is not in $\mathsf{CM}$.
The existence of such an $\mathbb{R}$-mouse follows from the existence of a proper class of strong cardinals.
The two aforementioned results provide tools for studying the internal structure of $\mathsf{CM}$. 

Here, our goal is to provide a third method for studying the internal structure of $\mathsf{CM}$.
Our method is perhaps more traditional: we identify a canonical iterable inner model such that $\mathsf{CM}$ can be generically added to it via genericity iterations.
This is similar to what is being done in inner model theory: e.g.\ $L(\mathbb{R})$ can be realized as the derived model, via $\mathbb{R}$-genericity iterations, of the minimal active mouse with infinitely many Woodin cardinals.
The following theorem, which will be proved as 
\cref{Det_in_Chang_in_any_gen_ext}, summarizes our work.

\begin{thm}
Let $(\mathcal{P}, \Sigma)$ be an excellent least branch hod pair\footnote{See \cite{comparison_mouse_pairs} for the definition of a least branch hod pair and see \cite{CoveringChang} for the definition of excellence.} such that
\begin{itemize}
    \item $\mathrm{Code}(\Sigma)$ is universally Baire, and
    \item $\mathcal{P}\models\mathsf{ZFC}$ $+$ ``there is a Woodin cardinal that is a limit of Woodin cardinals.''
\end{itemize}
Then $\mathsf{CM}\models\mathsf{AD}^+$ in any set generic extensions.
\end{thm}

The second author recently showed that the existence of an excellent least branch hod pair with a Woodin limit of Woodin cardinals is consistent relative to a measurable cardinal above a Woodin limit of Woodin cardinals, but the proof has not been written up yet.
We conjecture that some large cardinal hypothesis directly implies the existence of such a least branch hod pair.

The main theorem in this paper is a direct consequence of $\Sigma^2_1$ reflection proved as \cref{Sigma^2_1-reflection}, together with the main theorem of \cite{CoveringChang} stated as \cref{Chang_over_DM} in this paper.
The $\Sigma^2_1$ reflection is a completely fine-structure free result and our proof does not require any deep knowledge of inner model theory.
One can find in \cite{Farah} all definitions of basic notions used in this paper.

\subsection*{Acknowledgments}

The first author thanks Daisuke Ikegami to tell him about Woodin's old results on the Chang model.
The first author thanks John Steel for offering insightful remarks and comments when the first author gave talks on this topic at Core Model Seminar in January 2023.
The first author is grateful to Ernest Schimmerling and Benjamin Siskind for organizing the seminar and affording him the opportunity to present his work.
The first author was supported by the Elise Richter grant number V844 and the International Project grant number I6087 of the Austrian Science Fund (FWF). The second author was supported by NCN Grant WEAVE-UNISONO, Id: 567137.

\section{$\Sigma^2_1$ reflection of the Chang model}

\begin{thm}\label{Sigma^2_1-reflection}
Suppose that $V$ is closed under sharps.
Let $M$ be a countable transitive model of $\mathsf{ZFC}$.
Let $\vec{E}\in M$ be a sequence of extenders in $M$ and let $\delta$ be a cardinal of $M$ such that
\[
M\models\text{``$\delta$ is a Woodin cardinal that is a limit of Woodin cardinals witnessed by $\vec{E}$.''}
\]
Suppose that $(M, \vec{E})$ has an $\omega_1$-iteration strategy\footnote{In this paper, iterability always means normal iterability.} $\Sigma$ such that $\mathrm{Code}(\Sigma)$ is $<(2^{\aleph_0})^+$-universally Baire.
Let $g\subset\mathrm{Col}(\omega, <\delta)$ be $M$-generic.
Then, for any $\Sigma^2_1$-sentence $\phi$,
\[
\mathsf{CM}\models\phi\Longrightarrow (\mathsf{CM}\uphar\omega_2)^{M[g]}\models\phi.
\]
\end{thm}

\begin{proof}
Suppose that $\mathsf{CM}\models\phi$.
Let $\alpha_0, \beta_0\geq\omega$ be such that $L_{\beta_0}({}^{\omega}{\alpha_0})\models\phi$.
By Skolem hull argument, we may assume that $\lvert\max\{\alpha_0, \beta_0\}\rvert\leq 2^{\aleph_0}$.
Now we force $\mathsf{CH}$ via $\mathrm{Add}(\omega_1, 1)$.
As this is a countable closed forcing of size continuum, $\mathsf{CM}$ does not change and the assumption of the theorem is preserved.
Therefore, we may assume that $\mathsf{CH}$ holds in $V$.
Let $A\subset\mathbb{R}$ code a bijection $f\colon\omega_1\to\alpha_0$ and the structure $L_{\beta_0}({}^{\omega}\alpha_0)$.
Note that $(A, \mathbb{R})^{\#}$ exists.

By $\mathsf{CH}$, we can fix an enumeration $\langle r_{\xi}\mid\xi<\omega_1\rangle$ of $\mathbb{R}$.
For any $\gamma<\omega_1$ and $B\subset\mathbb{R}$, we write $B\uphar\gamma=B\cap\{r_{\xi}\mid\xi<\gamma\}$.
We make use of Woodin's extender algebra at a Woodin cardinal $\delta$ with $\delta$ many generators defined in $(M, \vec{E})$, denoted by $\mathsf{EA}_{\delta}^{M, \vec{E}}$.
We use the fact that $M\models$``$\mathsf{EA}_{\delta}^{M, \vec{E}}$ has $\delta$-c.c.'' and the following lemma.

\begin{lem}[{\cite[Theorem 4.6]{Farah}}]\label{extender_algebra}
Let $M$ be a countable transitive model of $\mathsf{ZFC}$.
Let $\vec{E}\in M$ be a sequence of extenders in $M$ and let $\delta$ be a cardinal of $M$ such that
\[
M\models\text{``$\delta$ is a Woodin cardinal that is a limit of Woodin cardinals witnessed by $\vec{E}$.''}
\]
Suppose that $(M, \vec{E})$ is $\omega_1+1$-iteration straetgy $\Sigma$.
Also, let $A\subset\mathbb{R}$ be such that $(A, \mathbb{R})^{\#}$ exists.
Then there is an iteration map $i\colon M\to N$ via a normal iteration tree $\mathcal{T}$ on $(M, \vec{E})$ of countable length according to $\Sigma$, and $N$-generic $g\subset i(\mathsf{EA}^{M, \vec{E}}_{\delta})$ such that 
\begin{enumerate}
    \item $L(A\uphar i(\delta), \mathbb{R}\uphar i(\delta))\prec L(A, \mathbb{R})$,
    \item $A\uphar i(\delta)\in N[g]$,
    \item $\mathbb{R}^{N[g]}=\mathbb{R}\upharpoonright i(\delta)$,
    \item for any $x\in\mathbb{R}^{N[g]}$, there is a poset $\mathbb{P}\in V^N_{i(\delta)}$ such that $x$ is $\mathbb{P}$-generic over $N$, and
    \item $i(\delta)=\omega_1^{N[g]}=\sup\{\omega_1^{N[x]}\mid x\in\mathbb{R}^{N[g]}\}$.
\end{enumerate}
\end{lem}
\begin{rem}
We need $(A, \mathbb{R})^{\#}$ only for the first clause.
\end{rem}

Note that our $(M, \vec{E})$ is $\omega_1+1$-iterable by $<(2^{\aleph_0})^{+}$-universally Baireess of its iteration strategy.
We apply \cref{extender_algebra} to our $A$ to get an elementary embedding $i\colon M\to N$ and $N$-generic $g\subset i(\mathsf{EA}^{M, \vec{E}}_{\delta})$ satisfying the above clauses (1)--(5).
By (1), $\mathbb{R}^{L(A\uphar i(\delta), \mathbb{R}\upharpoonright i(\delta))}=\mathbb{R}\upharpoonright i(\delta)$ and
\begin{multline*}
L(A\uphar i(\delta), \mathbb{R}\upharpoonright i(\delta))\models \exists\alpha, \beta\; [A\uphar i(\delta)\text{ codes the bijection } f\uphar i(\delta)\colon i(\delta)\to\alpha\\
\text{ and the model }L_{\beta}({}^{\omega}\alpha)\text{ of }\phi].
\end{multline*}
We fix ordinals $\alpha$ and $\beta$ witnessing this.

\begin{claim}\label{first_claim}
$({}^{\omega}\alpha)^{L(A\uphar i(\delta), \mathbb{R}\upharpoonright i(\delta))}=({}^{\omega}\alpha)^{N[g]}$.
\end{claim}
\begin{proof}
Since $A\uphar i(\delta), \mathbb{R}\uphar i(\delta)\in N[g]$ by (2) and (3), $({}^{\omega}\alpha)^{L(A\uphar i(\delta), \mathbb{R}\upharpoonright i(\delta))}\subset N[g]$.
To show that $({}^{\omega}\alpha)^{N[g]}\subset L(A\uphar i(\delta), \mathbb{R}\uphar i(\delta))$, let $x\in({}^{\omega}\alpha)^{N[g]}$.
Note that $f\uphar i(\delta)\colon i(\delta)\to\alpha$ is the bijection coded by $A\uphar i(\delta)$. 
Set
\[
x_f:=\langle f^{-1}(x(n)) \mid n<\omega\rangle\in ({}^{\omega}i(\delta))^{N[g]}.
\]
Since $i(\delta)=\omega_1^{N[g]}$ by (5), $x_f$ can be coded into a real in $N[g]$.
As $\mathbb{R}^{N[g]}=\mathbb{R}\upharpoonright i(\delta)$ by (3), $x_f$ is in $L(A\uphar i(\delta), \mathbb{R}\upharpoonright i(\delta))$.
Therefore, $x\in L(A\uphar i(\delta), \mathbb{R}\upharpoonright i(\delta))$.
\end{proof}

\begin{claim}\label{second_claim}
$({}^{\omega}\alpha)^{N[g]}=({}^{\omega}\alpha)^{N(\mathbb{R}\upharpoonright i(\delta))}$.
\end{claim}
\begin{proof}
We repeat almost the same proof of \cref{first_claim}.
It is enough to show that $({}^{\omega}\alpha)^{N[g]}\subset N(\mathbb{R}\upharpoonright i(\delta))$, so let $x\in ({}^{\omega}\alpha)^{N[g]}$.
Because $N[g]\models\lvert\alpha\rvert=i(\delta)$ and $N\models$``$i(\mathsf{EA}^{M, \vec{E}}_{\delta})$ has the $i(\delta)$-c.c.,'' $N\models\lvert\alpha\rvert=i(\delta)$.
Let $\overline{f}\colon i(\delta)\to\alpha$ be a bijection in $N$ and set
\[
x_{\overline{f}}:=\langle \overline{f}^{-1}(x(n)) \mid n<\omega\rangle\in {}^{\omega} i(\delta).
\]
Since $i(\delta)=\omega_1^{N[g]}$, $x_{\overline{f}}$ can be coded into a real in $N[g]$.
As $\mathbb{R}^{N[g]}=\mathbb{R}\uphar i(\delta)$ by (3), $x_{\overline{f}}\in N(\mathbb{R}\uphar i(\delta))$ and thus $x\in N(\mathbb{R}\uphar i(\delta))$.
\end{proof}

Now we use characterization of the symmetric reals in a collapsing extension.
We say that $X\subset\mathbb{R}$ is closed under finite sequences if for each finite subset of $X$, there is a single real in $X$ that recursively codes it.

\begin{lem}[{\cite[Lemma 3.1.5]{StationaryTower}}]\label{symmetric_reals}
Let $N$ be a transitive model of $\mathsf{ZFC}$ and let $\eta$ be a strong limit cardinal in $N$.
Let $X\subset\mathbb{R}$ be a set of reals closed under finite sequences such that each real in $X$ is generic over $N$ and $X=\mathbb{R}\cap N(X)$.
Then $X=\mathbb{R}^*_h :=\bigcup_{\xi<\delta}\mathbb{R}^{N[h\uphar\xi]}$ for some $N$-generic $h\subset\mathrm{Col}(\omega, <\eta)$ if and only if the following hold:
\begin{itemize}
    \item for any $x\in X$, there is a poset $\mathbb{P}\in V^{N}_{\eta}$ such that $x$ is $\mathbb{P}$-generic over $N$, and
    \item $\eta=\sup\{\omega_1^{N[x]}\mid x\in X\}$.
\end{itemize}
\end{lem}

We apply this lemma to $X=\mathbb{R}\uphar i(\delta)$ and $\eta=i(\delta)$.
The assumption on $X$ holds in our case because $\mathbb{R}\uphar i(\delta) = \mathbb{R}^{N(\mathbb{R}\uphar i(\delta))}=\mathbb{R}^{N[g]}$ by (3).
By (4) and (5), there is an $N$-generic $h\subset\mathrm{Col}(\omega, <i(\delta))$ such that $\mathbb{R}\uphar i(\delta) = \mathbb{R}^*_h = \mathbb{R}^{N[h]}$.

\begin{claim}\label{third_claim}
$({}^{\omega}\alpha)^{N(\mathbb{R}\upharpoonright i(\delta))}=({}^{\omega}\alpha)^{N[h]}$.
\end{claim}
\begin{proof}
Since $\mathbb{R}^{N[h]}=\mathbb{R}\upharpoonright i(\delta)$, the same proof as \cref{second_claim} works out.
\end{proof}

By \cref{first_claim,second_claim,third_claim}, $(L_{\beta}({}^{\omega}\alpha))^{N[h]}=(L_{\beta}({}^{\omega}\alpha))^{L(A\uphar i(\delta), \mathbb{R}\upharpoonright i(\delta))}$, which is a structure coded by $A\uphar i(\delta)$.
Therefore, $(L_{\beta}({}^{\omega}\alpha))^{N[h]}\models\phi$.
By the homogeneity of $\mathrm{Col}(\omega, <i(\delta))$,
\[
N\models \exists\alpha, \beta < i(\delta)^+ \; \Vdash_{\mathrm{Col}(\omega, <i(\delta))} \text{``}L_{\beta}({}^{\omega}\alpha)\models\phi.\text{''}
\]
By elementarity of $i\colon M\to N$,
\[
M\models \exists\alpha, \beta < \delta^+ \; \Vdash_{\mathrm{Col}(\omega, <\delta)}\text{``}L_{\beta}({}^{\omega}\alpha)\models\phi.\text{''}
\]
Therefore, for any $M$-generic $g\subset\mathrm{Col}(\omega, <\delta)$, we have $(\mathsf{CM}\uphar\omega_2)^{M[g]}\models\phi$.
\end{proof}

\section{Determinacy in the Chang model}

First, recall the derived model theorem due to Woodin.
Suppose that $\delta$ is a limit of Woodin cardinals and let $g\subset\mathrm{Col}(\omega, <\delta)$ be $V$-generic.
In $V[g]$, we let $\mathbb{R}^*_g = \bigcup_{\xi<\delta} \mathbb{R}^{V[g\vert\xi]}$ and
\[
\Gamma^*_g = \{A^*_g\subset\mathbb{R}^*_g\mid\exists\xi<\delta( A\subset\mathbb{R}^{V[g\uphar\xi]}\land V[g\uphar\xi]\models\text{``$A$ is $<\delta$-universally Baire''})\}.
\]
Here, we write $A^*_g = \bigcup_{\xi<\eta<\delta} A^{g\uphar\eta}$, where $A^{g\uphar\eta}$ is the canonical extension of $A$ in $V[g\uphar\eta]$ via its $<\delta$-universally Baire representation.
The model $L(\Gamma^*_g, \mathbb{R}^*_g)$ is called the derived model at $\delta$ computed in $V[g]$.
Woodin showed that $L(\Gamma^*_g, \mathbb{R}^*_g)\models\mathsf{AD}^+$.\footnote{See \cite{derived_model_theorem} for the proof of this.}

In \cite{CoveringChang}, the second author constructed a new model of determinacy extending a derived model in a collapsing extension of a hod mouse.

\begin{thm}[Sargsyan, \cite{CoveringChang}]\label{Chang_over_DM}
Suppose that $(\mathcal{P}, \Sigma)$ is an excellent least branch hod pair such that $\mathcal{P}\models\mathsf{ZFC}$ and that $\delta$ is a limit of Woodin cardinals such that either $\delta$ is regular or $\cf(\delta)$ is not a measurable cardinal in $\mathcal{P}$.
Let $g\subset\mathrm{Col}(\omega, <\delta)$ be $\mathcal{P}$-generic.
Let $L(\Gamma^*_g, \mathbb{R}^*_g)$ be the derived model at $\delta$ computed in $\mathcal{P}[g]$.
Then, in $\mathcal{P}[g]$, there is a transitive model $M$ of $\mathsf{ZFC}-\mathsf{Power}$ such that
\begin{itemize}
    \item $M\cap\ord=\omega_2=(\delta^+)^{\mathcal{P}}$,
    \item $M$ has a largest cardinal $\kappa$ such that if $\cf^{\mathcal{P}}(\delta)>\omega$ then $\mathrm{cf}^{\mathcal{P}[g]}(\kappa)>\omega$,
    \item $L(M, \cup_{\alpha<\kappa}{}^{\omega}\alpha, \Gamma^*_g, \mathbb{R}^*_g)\cap\power(\mathbb{R})=\Gamma^*_g$, and
    \item $L(M, \cup_{\alpha<\kappa}{}^{\omega}\alpha, \Gamma^*_g, \mathbb{R}^*_g)\models\mathsf{AD}^+$.\footnote{This follows from the third clause by the derived model theorem.}
\end{itemize}
\end{thm}

\begin{rem}
This remark is not necessary for our proof, but more is shown in \cite{CoveringChang}.
Let $(\mathcal{P}, \Sigma), \delta$ and $g$ be as above.
In $\mathcal{P}[g]$, we define $\mathcal{M}_{\infty}$ as the direct limit of all non-dropping iterates of $\mathcal{P}\vert(\delta^+)^{\mathcal{P}}$ via normal iteration trees of length $<\delta$ that are based on $\mathcal{P}\vert\delta$ and in $\mathcal{P}[g\uphar\xi]$ for some $\xi<\delta$.
Also let $\delta_{\infty}$ be the direct limit image of $\delta$ in $\mathcal{M}_{\infty}$.
Then the conclusion of \cref{Chang_over_DM} holds for $M=\mathcal{M}_{\infty}$ and $\kappa=\delta_{\infty}$.
When working in $\mathcal{P}[g]$, we write
\[
\mathsf{CDM}=L(\mathcal{M}_{\infty}, \cup_{\alpha<\delta_{\infty}}{}^{\omega}\alpha, \Gamma^*_g, \mathbb{R}^*_g).
\]
Here, $\mathsf{CDM}$ means ``the Chang model over the derived model.''

It is known that $\mathsf{CDM}$ satisfies stronger determinacy axiom.
Steel showed in \cite{comparison_mouse_pairs} that the derived model of an excellent least branch hod mouse at a limit of Woodin cardinals satisfies $\mathsf{AD}_{\mathbb{R}}$.
Furthermore, the authors showed in \cite{dm_of_self_it} that the derived model of a hod mouse at an inaccessible limit of Woodin cardinals satisfies ``$\mathsf{AD}_{\mathbb{R}}+\Theta$ is regular.''
Since $\mathsf{CDM}$ and the derived model have the same sets of reals, it follows that $\mathsf{CDM}$ is always a model of $\mathsf{AD}_{\mathbb{R}}$ and possibly a model of ``$\mathsf{AD}_{\mathbb{R}}+\Theta$ is regular.''
\end{rem}

Now we are ready to show the main theorem in this paper.

\begin{thm}\label{Det_in_Chang}
Suppose that $V$ is closed under sharps.
Let $(\mathcal{P}, \Sigma)$ be an excellent least branch hod pair such that
\begin{itemize}
    \item $\mathrm{Code}(\Sigma)$ is $<(2^{\aleph_0})^+$-universally Baire, and
    \item $\mathcal{P}\models\mathsf{ZFC}$ $+$ ``there is a Woodin cardinal that is a limit of Woodin cardinals.''
\end{itemize}
Then $\mathsf{CM}\models\mathsf{AD}^+$.
\end{thm}
\begin{proof}
Let $M$ be a model of $\mathsf{ZFC}-\mathsf{Power}$ in the conclusion of \cref{Chang_over_DM}.
Also, let $\kappa$ be the largest cardinal of $M$.

\begin{claim}\label{CDM_contains_CM}
In $\mathcal{P}[g]$,
$\mathsf{CM}\uphar\omega_2\subset L(M, \cup_{\alpha<\kappa}{}^{\omega}\alpha, \Gamma^*_g, \mathbb{R}^*_g)$.
\end{claim}
\begin{proof}
We argue in $\mathcal{P}[g]$.
Let $\alpha<\omega_2=\ord\cap M$ and let $x\in {}^{\omega}\alpha$.
We may assume that $\alpha\geq\kappa$.
As $\kappa$ is the largest cardinal in $M$, we can fix a bijection $f\colon \kappa\to\alpha$ in $M$.
Set $x_f:=\langle f^{-1}(x(n)) \mid n<\omega\rangle\in {}^{\omega}\kappa$.
Since $\kappa$ has uncountable cofinality, $x_f\in {}^{\omega}\xi$ for some $\xi<\kappa$.
Then $x_f$ is an element of $L(M, \cup_{\alpha<\kappa}{}^{\omega}\alpha, \Gamma^*_g, \mathbb{R}^*_g)$ and thus so is $x$.
\end{proof}

Since $(\mathsf{CM}\uphar\omega_2)^{\mathcal{P}[g]}$ contains all reals in $\mathcal{P}[g]$, it follows from \cref{Chang_over_DM} that $(\mathsf{CM}\uphar\omega_2)^{\mathcal{P}[g]}\models\mathsf{AD}^+$.
As the negation of $\mathsf{AD}$ is $\Sigma^2_1$, $\mathsf{CM}\models\mathsf{AD}$ by \cref{Sigma^2_1-reflection}.
Moreover, under $\mathsf{AD}$, the negation of $\mathsf{AD}^+$ is equivalent to a $\Sigma^2_1$ sentence.
So again, \cref{Sigma^2_1-reflection} implies that $\mathsf{CM}\models\mathsf{AD}^+$.
\end{proof}

\begin{cor}\label{Det_in_Chang_in_any_gen_ext}
Let $(\mathcal{P}, \Sigma)$ be an excellent least branch hod pair such that
\begin{itemize}
    \item $\mathrm{Code}(\Sigma)$ is universally Baire, and
    \item $\mathcal{P}\models\mathsf{ZFC}$ $+$ ``there is a Woodin cardinal that is a limit of Woodin cardinals.''
\end{itemize}
Then $\mathsf{CM}\models\mathsf{AD}^+$ in any set generic extensions.
\end{cor}
\begin{proof}
Since $\mathcal{P}$ has a universally Baire iteration strategy, it is fully iterable.
It follows that $V$ is closed under sharps.
Then the corollary follows from \cref{Det_in_Chang} and the observation that in any set generic extensions $V[g]$, letting $\Sigma^g$ be the canonical extension of $\Sigma$ in $V[g]$, $(\mathcal{P}, \Sigma^g)$ is still an excellent least branch hod pair.
\end{proof}

Mitchell's result from \cite{small_sharp_for_Chang} that \cref{Det_in_Chang} cannot be strengthened to $\mathsf{CM}\models\mathsf{AD}_{\mathbb{R}}$.
This is because $\mathsf{AD}_{\mathbb{R}}$ implies that Mitchell's $\mathbb{R}$-mouse exists, and hence it is not in $\mathsf{CM}$.
In fact, it follows from Mitchell's theorem that the theory $\mathsf{AD}^{+} +\theta_0<\Theta$ cannot hold in $\mathsf{CM}$.
However, Mitchell's proofs require knowledge of the sophisticated forcing machinery developed by Gitik. Here we state a simpler albeit weaker result showing that $\mathsf{CM}$ cannot be too large.

\begin{thm}
Under the same assumption of \cref{Det_in_Chang}, $\Sigma\notin\mathsf{CM}$.
\end{thm}
\begin{proof}
Suppose that $\Sigma\in\mathsf{CM}$.
For the same reason as before, we may assume that $\mathsf{CH}$ holds.
Then the argument for the proof of \cref{Sigma^2_1-reflection} shows that there are
\begin{itemize}
\item an iteration map $i\colon \mathcal{P}\to \mathcal{Q}$ via a normal iteration tree $\mathcal{T}$ on $\mathcal{P}$ of countable length according to $\Sigma$,
\item ordinals $\alpha, \beta<(i(\delta)^+)^{\mathcal{Q}}$, and
\item $\mathcal{Q}$-generic $h\subset\mathrm{Col}(\omega, <i(\delta))$
\end{itemize}
such that $\Sigma\uphar\mathsf{HC}^{\mathcal{Q}[h]}\in (L_{\beta}({}^{\omega}{\alpha}))^{\mathcal{Q}[h]}$.
Then in $\mathcal{Q}[h]$, \cref{CDM_contains_CM} and \cref{Chang_over_DM} implies that $\mathrm{Code}(\Sigma\uphar\mathsf{HC}^{\mathcal{Q}[h]})$ is $<i(\delta)$-universally Baire.

Now we can get a contradiction in a rather standard way in inner model theory.
Let $(T, U)\in\mathcal{Q}[h]$ be a pair of trees on $\omega\times\ord$ witnessing that $\mathrm{Code}(\Sigma\uphar\mathsf{HC}^{\mathcal{Q}[h]})$ is $<i(\delta)$-universally Baire.
Let $\gamma<i(\delta)$ be such that $(T, U)\in\mathcal{Q}[h\uphar\gamma]$.
Then for any ordinal $\eta$ such that $\gamma<\eta<i(\delta)$,
\[
\Sigma\uphar V^{\mathcal{Q}[h\uphar\eta]}_{i(\delta)}\in\mathcal{Q}[h\uphar\eta].
\]
Now pick $\eta$ as the least regular cardinal of $\mathcal{Q}$ above $\gamma$ in the range of $i\colon\mathcal{P}\to\mathcal{Q}$.
This is possible because $i(\delta)=\sup i[\delta]$.
Let $\xi$ be the least ordinal in the last branch of $\mathcal{T}$ such that the critical point of the iteration map $i^{\mathcal{T}}_{\xi, \infty}\colon \mathcal{M}^{\mathcal{T}}_{\xi}\to\mathcal{Q}$ is at least $\eta$.
Note that $\crit(i^{\mathcal{T}}_{\xi, \infty})>\eta$ since $\eta$ is in the range of $i$.
Let $\overline{\eta}:=(i^{\mathcal{T}}_{\xi, \infty})^{-1}(\eta)=i^{-1}(\eta)$.
Then $\overline{\eta}$ is regular in $\mathcal{P}$ and thus $\eta=\sup i^{\mathcal{T}}_{\xi, \infty}[\overline{\eta}]$.
Since $\Sigma\uphar V^{\mathcal{Q}[h\uphar\eta]}_{i(\delta)}\in\mathcal{Q}[h\uphar\eta]$, we have $\mathcal{T}\uphar\xi+1\in\mathcal{Q}[h\uphar\eta]$ and $\overline{\eta}$ is countable in $\mathcal{Q}[h\uphar\eta]$.
It follows that $\eta$ has countable cofinality in $\mathcal{Q}[h\uphar\eta]$.
Then $\eta$ has countable cofinality in $\mathcal{Q}$ too, which contradicts the regularity of $\eta$ in $\mathcal{Q}$.
\end{proof}

Many questions remain open. Can the generic invariance of the theory of $\textsf{CM}$ be proved from the hypothesis of \cref{Det_in_Chang}? 
Woodin introduced various kinds of generalized Chang models in \cite{DetGen}.
One of such examples is the Chang model augmented with the club filters on $\power_{\omega_1}({}^{\omega}\alpha)$ for all ordinals $\alpha$.
Woodin showed that this version of generalized Chang model satisfies $\mathsf{AD}^+ +$ ``$\omega_1$ is supercompact,'' assuming the existence of a proper class of Woodin limits of Woodin cardinals.\footnote{This result is also mentioned in \cite{spct_of_omega_1}.}
Is it possible to prove Woodin's results by the methods of \cref{Det_in_Chang}? We will address such questions in future publications. 

\bibliographystyle{amsalpha.bst}
\providecommand{\bysame}{\leavevmode\hbox to3em{\hrulefill}\thinspace}
\providecommand{\MR}{\relax\ifhmode\unskip\space\fi MR }
\providecommand{\MRhref}[2]{%
  \href{http://www.ams.org/mathscinet-getitem?mr=#1}{#2}
}
\providecommand{\href}[2]{#2}

\end{document}